\documentstyle[10pt,amscd,xypic,amssymb]{amsart}

\xyoption{all}
\CompileMatrices

\newcommand{\nc}{\newcommand}

\makeatletter
\@addtoreset{equation}{section}
\makeatother

\newenvironment{proof}{{\noindent \textbf{Proof}\,\,}}{\hspace*{\fill}$\Box$\medskip}

\newtheorem{theorem}[subsection]{Theorem}

\newtheorem{proposition}[subsection]{Proposition}

\theoremstyle{definition}

\theoremstyle{remark}

\nc{\fa}{{\mathfrak{a}}}
\nc{\fb}{{\mathfrak{b}}}
\nc{\fg}{{\mathfrak{g}}}
\nc{\fh}{{\mathfrak{h}}}
\nc{\fj}{{\mathfrak{j}}}
\nc{\fn}{{\mathfrak{n}}}
\nc{\fm}{{\mathfrak{m}}}
\nc{\fu}{{\mathfrak{u}}}
\nc{\fp}{{\mathfrak{p}}}
\nc{\fr}{{\mathfrak{r}}}
\nc{\ft}{{\mathfrak{t}}}
\nc{\fsl}{{\mathfrak{sl}}}
\nc{\fgl}{{\mathfrak{gl}}}
\nc{\hsl}{{\widehat{\mathfrak{sl}}}}
\nc{\hgl}{{\widehat{\mathfrak{gl}}}}
\nc{\hg}{{\widehat{\mathfrak{g}}}}
\nc{\chg}{{\widehat{\mathfrak{g}}}{}^\vee}
\nc{\hn}{{\widehat{\mathfrak{n}}}}
\nc{\chn}{{\widehat{\mathfrak{n}}}{}^\vee}

\nc{\wGL}{{\widehat{GL}^+}}

\nc{\BA}{{\mathbb{A}}}
\nc{\BC}{{\mathbb{C}}}
\nc{\BM}{{\mathbb{M}}}
\nc{\BN}{{\mathbb{N}}}
\nc{\BF}{{\mathbb{F}}}
\nc{\BP}{{\mathbb{P}}}
\nc{\BR}{{\mathbb{R}}}
\nc{\BZ}{{\mathbb{Z}}}

\nc{\C}{{\mathcal{C}}}
\nc{\pC}{{\mathbb{P}\C}}

\nc{\CA}{{\mathcal{A}}}
\nc{\CB}{{\mathcal{B}}}
\nc{\CE}{{\mathcal{E}}}
\nc{\CF}{{\mathcal{F}}}
\nc{\tCF}{{\widetilde{\CF}}}
\nc{\oCF}{{\overline{\CF}}}
\nc{\CG}{{\mathcal{G}}}
\nc{\CL}{{\mathcal{L}}}
\nc{\CM}{{\mathcal{M}}}
\nc{\CH}{{\mathcal{H}}}
\nc{\CN}{{\mathcal{N}}}
\nc{\CO}{{\mathcal{O}}}
\nc{\CP}{{\mathcal{P}}}
\nc{\CR}{{\mathcal{R}}}
\nc{\CQ}{{\mathcal{Q}}}
\nc{\CS}{{\mathcal{S}}}
\nc{\CT}{{\mathcal{T}}}
\nc{\CU}{{\mathcal{U}}}
\nc{\CV}{{\mathcal{V}}}
\nc{\CW}{{\mathcal{W}}}

\nc{\tN}{{\widetilde{\CN}}}
\nc{\pN}{{\BP\widetilde{\CN}}}

\nc{\tT}{{\widetilde{T}}}

\nc{\fC}{{C}}
\nc{\fZ}{{Z}}
\nc{\wfZ}{{\widetilde{\fZ}}}

\nc{\od}{{\overline{d}}}
\nc{\rg}{{\textrm{R}\Gamma}}
\nc{\erg}{{\emph{R}\Gamma}}
\nc{\id}{{\textrm{Id}}}
\nc{\rhom}{{\textrm{RHom}}}

\nc{\tV}{{\widetilde{V}}}

\def\and{\textrm{ }\&\textrm{ }}

\begin{document}

\title[Push-forwards on Projective Towers]{\large{\textbf{Push-forwards on Projective Towers}}}

\author[Andrei Negut]{Andrei Negut}
\address{Harvard University, Department of Mathematics, Cambridge, MA 02138, USA}
\address{Simion Stoilow Institute of Mathematics, Bucharest, Romania}
\email{andrei.negut@@gmail.com}

\maketitle

\begin{abstract} In this paper we derive a simple and useful combinatorial formula for the push-forwards of cohomology classes down projective towers, in terms of the push-forwards down the individual steps in the tower. 
\end{abstract}

\section{Introduction}

\subsection{} Consider a proper map of algebraic varieties:

$$
\pi:X' \longrightarrow X.
$$
Pick any class $c\in A^*(X')$ and call it the \textbf{tautological class} of $\pi$. Relative to this choice, we can define the \textbf{Segre series} of $\pi$:

\begin{equation}
\label{eqn:segre}
s(\pi,u) = \pi_* \left( \frac 1{u - c} \right).
\end{equation}
This series in $u^{-1}$ has coefficients in $A^*(X)$, and it encodes the push-forwards of all powers of the tautological class $c$. \\

The terminology is motivated by the case when $X'=\BP_X \CV$ is the projectivization of a cone on $X$. In this case, we let $c = c_1(\CO(1))$ and the above notion coincides (up to normalization) with the Segre class introduced by Fulton in \cite{Ful}. In the particular case when $\CV$ is a vector bundle, the Segre series equals the inverse of the (properly renormalized) Chern polynomial of $\CV$. \\

\subsection{} The main subject of this paper are \textbf{projective towers}, namely compositions of proper maps of algebraic varieties:

\begin{equation}
\label{eqn:tower}
\pi: X_k \stackrel{\pi^k}\longrightarrow X_{k-1} \stackrel{\pi^{k-1}}\longrightarrow ... \stackrel{\pi^2}\longrightarrow X_1 \stackrel{\pi^1}\longrightarrow X_0.
\end{equation}
As before, pick $c_i \in A^*(X_i)$ and call them the \textbf{tautological classes} of the tower. We want to encode the push-forwards of these tautological classes under $\pi$, and the reasonable way to do this is to define the \textbf{Segre series} of the tower as:

\begin{equation}
\label{eqn:segree} s(\pi,u_1,...,u_k) = \pi_*\left(\frac 1{u_1 - c_1} \cdot ... \cdot \frac 1{u_k - c_k} \right)
\end{equation}
We suppress the obvious pull-back maps to $X_k$, and hope that this will cause no confusion. \\

\subsection{} 
\label{sub:stat}

One of the main technical results of \cite{Yang} involves studying a particular projective tower \eqref{eqn:tower}. One needs to derive a closed formula for the Segre series of the whole tower $\pi$ from the Segre series of the individual maps $\pi^i$. The assumption we make on these individual Segre series is that:

\begin{equation}
\label{eqn:ass}
s(\pi^i,u) = \prod_{m_i} Q_{m_i}(u + m_i^{i-1} c_{i-1} + ... + m_i^{1} c_{1}),
\end{equation}
for some series $Q_{m_i}$ with coefficients in $A^*(X_0)$, where the product goes over finitely many vectors $m_i = (m_i^1,...,m_i^{i-1})$ of integers. This assumption will be motivated in section \ref{sub:vect}, based on the particular example of a tower of projective bundles. Then our main Theorem \ref{thm:main} below implies that:

\begin{equation}
\label{eqn:main}
s(\pi,u_1,...,u_k) = \left[\prod_{i=1}^k \prod_{m_i} Q_{m_i}(u_i + m_i^{i-1} u_{i-1} + ... + m_i^1 u_1) \right]_-
\end{equation}
The notation $[ \dots ]_-$ means that we expand each $Q_{m_i}$ in non-negative powers of $u_{i-1},...,u_1$, and then we only keep the monomials with all negative exponents in the resulting formula. \\

\subsection{} The basic idea, naturally, is to successively push forward the tautological classes from $X_k$ to $X_{k-1}$ to $\dots$ to $X_1$ to $X_0$, and assumption \eqref{eqn:ass} provides the means for this recursion. However, if one carried out this procedure straightforwardly, one would not obtain a closed formula. The reason why formula \eqref{eqn:main} looks so nice is that we are adding terms with non-negative exponents, only to get rid of them when we apply $[\dots]_-$ at the very end. \\

This closed formula is very useful in the papers \cite{Yang} and \cite{Laumon}. In the present note, we will present a baby case of the main technical computation of these papers: we will rederive a closed formula for integrals on the complete flag variety of vector subspaces of a fixed vector space. \\

I would like to thank Mircea Mustata, Aaron Silberstein and Aleksander Tsymbaliuk for useful discussions. \\

\section{Tautological Classes} 
\label{sec:taut}

\subsection{} 
\label{sub:vect}

Consider the special case of \eqref{eqn:tower} where $X_{i} = \BP_{X_{i-1}} \CV_{i}$ for some vector bundle $\CV_i$ of rank $r_i$ on $X_{i-1}$, and $c_i=c_1(\CO_i(1))$ is the first Chern class of the tautological line bundle. It is well-known (\cite{Ful}, Section 3.2) that the individual Segre classes are equal to the inverse Chern classes:

\begin{equation}
\label{eqn:segrechern}
s(\pi^{i},u) = c^{-1}(\CV_i,u) \qquad \textrm{ where} \qquad c(\CV_i,u)= u^{r_i} \cdot \sum_{k=0}^{r_i} u^{-k} c_k(\CV_i).
\end{equation}
The above Chern classes only depend on the class of $\CV_i$ in the $K-$theory of $X_{i-1}$. We will make the following assumption on this class:

$$
[\CV_i] = \sum_{m_i} [V_{m_i}] \otimes [\CO_1(m_i^1)] \otimes ... \otimes [\CO_{i-1}(m_i^{i-1})],
$$
in $K-$theory, where the sum goes over finitely many vectors $m_i = (m_i^1,...,m_i^{i-1})$ of integers and $[V_{m_i}] \in K(X_0)$ are arbitrary classes (we are suppressing the obvious pull-back maps, and hope that this will cause no confusion). In other words, we assume that in $K-$theory each $\CV_i$ is constructed by twisting bundles on the lower steps in the tower by various tautological line bundles. Then the Whitney sum formula tells us that:
$$
s(\CV_i,u) = \prod_{m_i} s(V_{m_i}, u + m_i^{i-1} c_{i-1} + ... + m_i^1 c_1),
$$
where the Segre series $s(V_{m_i},u)$ now have coefficients in $A^*(X_0)$. This setup justifies our assumption \eqref{eqn:ass}. \\

\subsection{} 
\label{sub:taut}

Let us now go to a general projective tower \eqref{eqn:tower} that satisfies assumption \eqref{eqn:ass}. Along with the variable $u_i$, for each $i \in \{1,...,k\}$ pick an extra set of variables $A_i$. Then our main result is the following theorem: \\

\begin{theorem} 
\label{thm:main}

We have the following relation:

$$
\pi_*\prod_{i=1}^k \left(\frac 1{u_i-c_i} \prod_{u \in A_i} \frac 1{u-c_i}\right) =  
$$

\begin{equation}
\label{eqn:taut}
=\left[\prod_{i=1}^k \prod_{m_i} Q_{m_i}(u_i + m_i^{i-1} u_{i-1} + ... + m_i^1 u_1) \prod_{u \in A_i} \frac 1{u-u_i} \right]_- 
\end{equation}
where we expand each $Q_{m_i}$ in non-negative powers of $u_{i-1},...,u_1$, and each $(u-u_i)^{-1}$ in non-negative powers of $u_i$. The notation $[\dots]_-$ means that we only keep the terms for which all the $u_i$'s and $u$'s have negative exponents. \\

\end{theorem}

Relation \eqref{eqn:main} is simply the case when all the $A_i$ are empty. Though the difference between \eqref{eqn:main} and \eqref{eqn:taut} is a purely formal manipulation of series, we are working with this more general format for the purposes of \cite{Yang}. \\

\begin{proof} For each $i$ between $0$ and $k$, let us define the quantity:

$$
Z_j = \pi^1_{*}...\pi^j_{*}\left[\prod_{i=1}^j \left(\frac 1{u_i-c_i} \prod_{u \in A_i} \frac 1{u-c_i}\right) \right. \cdot 
$$

$$
\cdot \left. \prod_{i=j+1}^k \prod_{m_i} Q_{m_i}(u_i + m_i^{i-1} u_{i-1} + ... + m_i^{j+1} u_{j+1} + m_i^j c_j + ... + m_i^1 c_1) \prod_{u \in A_i} \frac 1{u-u_i} \right]_-
$$
It is easy to see that $Z_k$ is the LHS and $Z_0$ is the RHS of \eqref{eqn:taut}. Therefore, to complete the proof of our theorem, we need to show that $Z_j = Z_{j-1}$, or in other words that:

$$
\pi^j_{*}\left[\left(\frac 1{u_j-c_j} \prod_{u \in A_j} \frac 1{u-c_j}\right) \cdot \prod_{i=j+1}^k \prod_{m_i} Q_{m_i}(u_i +  ... + m_i^j c_j + ... m_i^1 c_1) \right]_- = 
$$ 

\begin{equation}
\label{eqn:desire}
=\left[\prod_{u \in A_j} \frac 1{u-u_j}  \prod_{i=j}^k \prod_{m_i} Q_{m_i}(u_i + ... + m_i^j u_j + ... + m_i^1 c_1) \right]_-
\end{equation}
To prove this relation, it is enough to assume $Q_{m_i}(u) = u^{\alpha_{m_i}}$ and then the LHS becomes

$$
\pi^j_{*}\left[\sum_{\beta_j, \beta_u, \beta_{m_i}} c_j^{\beta_j + \sum \beta_u + \sum \beta_{m_i}} u_j^{-\beta_j-1} \prod_{u \in A_j} u^{-\beta_u-1} \prod_{i>j}^{m_i} (m_i^j)^{\beta_{m_i}} (u_i+...)^{\alpha_{m_i}-\beta_{m_i}} {\alpha_{m_i} \choose \beta_{m_i}}\right]_-
$$
where all the $\beta$'s range over the non-negative integers, $u$ ranges over $A_j$ and $i$ ranges over $\{j+1,...,k\}$. Now if we denote by $\gamma$ the exponent of $c_j$ and solve for $\beta_j$, the above becomes:

$$
\pi^j_{*}\left[\sum_{\gamma,\beta_u,\beta_{m_i}} c_j^\gamma u_j^{-\gamma-1} \prod_{u \in A_j} \left( u_j^{\beta_u} u^{-\beta_u-1}\right) \prod_{i>j}^{m_i} (m_i^j u_j)^{\beta_{m_i}} (u_i+...)^{\alpha_{m_i}-\beta_{m_i}} {\alpha_{m_i} \choose \beta_{m_i}} \right]_-
$$
The condition that $\beta_j \geq 0$, which was lost when we replaced it by the variable $\gamma$, is recovered by the condition $[\dots]_-$. Since $\gamma$ and the $\beta$'s sum independently, the above equals:

$$
\left[ \left( \pi_{j*} \frac 1{u_j-c_j} \right) \prod_{u \in A_j} \frac 1{u-u_j}  \prod_{i=j+1}^k \prod_{m_i} Q_{m_i}(u_i + ... + m_i^j u_j + ... + m_i^1 c_1) \right]_-
$$
Then if we replace $\pi_{j*} (u_j-c_j)^{-1}$ by \eqref{eqn:ass}, the above yields the RHS of \eqref{eqn:desire}, thus completing the proof. \\

\end{proof}

\section{A Basic Example}
\label{sec:flag}

Theorem \eqref{thm:main} works equally well if we replace Chow rings by cohomology rings. For a simple example, let us consider the variety $F$ of complete flags in $\BC^{k+1}$:

\begin{equation}
\label{eqn:flag}
V_1 \subset ... \subset V_{k} \subset \BC^{k+1},
\end{equation}
where $V_i$ is an $i-$dimensional subspace. On $F$, we will consider the universal vector bundle $\CV_i$ whose fiber over \eqref{eqn:flag} is $V_i$, and also the tautological line bundle: 

\begin{equation}
\label{eqn:radu}
\CO_i(1) = \CV_{k+2-i}/\CV_{k+1-i}.
\end{equation}
It is well known that $c_i=c_1(\CO_i(1))$ as $i \in \{1, ... , k\}$ generate the cohomology of $F$. Then we have the following result: \\

\begin{proposition} 
\label{prop:flag}

The following identity tells us how to integrate any cohomology class on $F$:

$$
\int_F \frac 1{(u_1 - c_1) ... (u_{k} - c_{k})} = (u_1...u_k)^{-k-1} \prod_{1\leq i < j \leq k}  (u_j-u_i) .
$$
$$$$
\end{proposition}

\begin{proof} If we let $F_i$ parametrize flags:

$$
V_{k+1-i} \subset ... \subset V_{k} \subset \BC^{k+1},
$$
where each $V_j$ still has dimension $j$, then $F_0 = \textrm{pt}$ and $F_{k} = F$. All these spaces fit into a projective tower:

$$
\pi: F_{k} \stackrel{\pi^{k}}\longrightarrow F_{k-1} \stackrel{\pi^{k-1}}\longrightarrow ... \stackrel{\pi^2}\longrightarrow F_1 \stackrel{\pi^1}\longrightarrow F_0 = \textrm{pt}.
$$
It is easy to see that $F_{i} = \BP_{F_{i-1}} (\CV_{k+2-i}^\vee)$, so this realizes the flag variety as a tower of projective bundles over the point. It's easy to see that $\CO_i(1)$ of this tower are precisely the line bundles \eqref{eqn:radu}, and therefore we have the following equality in the Grothendieck group of $F_i$:

$$
[\CV_{k+2-i}^\vee] = [\CO^{k+1}] - [\CO_1(-1)] - ... - [\CO_{i-1}(-1)],
$$
By the Whitney sum formula and \eqref{eqn:segrechern}, one therefore has:

$$
s(\pi^i,u) = \frac {(u-c_1)...(u-c_{i-1})}{u^{k+1}},
$$
Then \eqref{eqn:main} implies the desired result. We do not need the $[\dots]_-$ anymore, because all terms only consist of negative monomials already. \\

\end{proof}


\begin{thebibliography}{XXX}

\bibitem{Ful} Fulton W. {\em Intersection Theory}, Springer-Verlag, 1998

\bibitem{Yang} Negut A. {\em Yangians of $\hgl_n$ and Affine Laumon Spaces}, work in progress

\bibitem{Laumon} Negut A. {\em Affine Laumon Spaces and Integrable Systems}, work in progress

\end{thebibliography}
\end{document}